\documentclass{amsart}
\usepackage{amssymb,xcolor,hyperref}

\theoremstyle{plain}
\newtheorem*{mainthm}{Main Theorem}
\newtheorem{thm}{Theorem}[section]
\newtheorem{prop}[thm]{Proposition}

\newtheorem{lemma}[thm]{Lemma}
\newtheorem{conj}[thm]{Conjecture}

\theoremstyle{definition}
\newtheorem{definition}[thm]{Definition}

\newtheorem{remark}[thm]{Remark}
\newtheorem{examples}[thm]{Examples}
\newtheorem{notation}[thm]{Notation}
\newtheorem{example}[thm]{Example}
\numberwithin{equation}{section}

\newcommand{\Tor}{\mathrm{Tor}}
\newcommand{\HS}{\mathrm{HS}}
\newcommand{\HF}{\mathrm{HF}}

\newcommand{\Char}{\mathrm{char}}
\newcommand{\ch}{\mathrm{char}}
\newcommand{\codim}{\mathrm{codim}}

\newcommand{\initial}{\mathrm{in}}
\newcommand{\lex}{\mathrm{lex}}
\newcommand{\bideg}{\mathrm{bideg}}

\newcommand{\mm}{\mathfrak{m}}
\newcommand{\ff}{\mathfrak{f}}
\newcommand{\opp}{{\overline{\wp}}}

\newcommand{\bfd}{{\mathbf{d}}}
\newcommand{\obfd}{{\overline{\mathbf{d}}}}
\newcommand{\LL}{{\mathrm{L}}}
\newcommand{\ovS}{{\overline{S}}}

\title{On the Lex-Plus-Powers conjecture}

\author{Giulio Caviglia \and Alessio Sammartano}

\address[Giulio Caviglia]{Department of Mathematics, Purdue University, 150 N. University Street, West Lafayette, IN 47907, USA}

\email{gcavigli@purdue.edu}

\address[Alessio Sammartano]{Department of Mathematics, University of Notre Dame, 255 Hurley, Notre Dame, IN 46556, USA}

\email{asammart@nd.edu}


\begin{document}

\begin{abstract}
Let  $S$ be a polynomial ring over a field and $I\subseteq S$  a homogeneous  ideal containing a regular sequence of forms of degrees $d_1, \ldots, d_c$.
In this paper we prove the Lex-plus-powers Conjecture when the field has characteristic 0 for all regular sequences such that $d_i \geq \sum_{j=1}^{i-1} (d_j-1)+1$ for each $i$;
that is,
we show that the Betti table of $I$ is bounded above by the Betti table of the lex-plus-powers ideal of $I$.
\end{abstract}

\maketitle

\section{Introduction}

Let $S$ be a polynomial ring over a field.
The celebrated theorem of Macaulay 	\cite{M27} asserts the existence of a one-to-one correspondence 
between Hilbert functions of homogeneous ideals in $S$ and \emph{lex} ideals, i.e. ideals which in each degree are generated by an initial segment of monomials in the lexicographic order.
The result may be phrased equivalently in terms of bounds for the growth of the graded components of an ideal, 
or as the statement that lex ideals have the largest number of minimal generators $\beta_{0,j}$ allowed by their Hilbert function in each degree $j$.
An elegant generalization of this result is the Bigatti-Hulett-Pardue Theorem \cite{B93,H93,P96}, which states that in fact lex ideals have the largest possible graded Betti numbers  $\beta_{i,j}$  in every homological degree $i$ and internal degree $j$, yielding thus a unique maximal element in the poset of Betti tables for each Hilbert function.
A crucial tool in both theorems is the use of Gr\"obner deformations, which allow to build a flat family connecting an arbitrary homogeneous ideal to a monomial ideal fixed under the action of the Borel group.

In several geometric situations, 
related e.g. to questions about configurations of points in $\mathbb{P}^m$ or  Hilbert schemes of projective varieties other than $\mathbb{P}^m$,
it is desirable to have refinements of these two theorems which take into account not just the numerical data of an ideal, but more precise information about its structure. 
With these regards, there are two long standing conjectures on the graded invariants of a homogeneous ideal containing  a regular sequence of known degrees:
the Eisenbud-Green-Harris Conjecture and the Lex-plus-powers Conjecture.

The first conjecture was proposed in \cite{EGH93,EGH96}  with the aim of generalizing a classical theorem of  Castelnuovo as well as  Cayley-Bacharach-type theorems on zero-dimensional subschemes of $\mathbb{P}^m$.
Actually, the authors formulated a series of related conjectures; the most general statement may be given  as follows:

\begin{conj}[Eisenbud-Green-Harris]\label{ConjectureEGH}
Let $I\subseteq S=\Bbbk[x_1,\ldots, x_n]$ be a homogeneous ideal containing a regular sequence of degrees $d_1\leq  \cdots \leq d_c$. 
Then there exists a lex ideal $L\subseteq S$ such that $I$ has the same Hilbert function as $L+(x_1^{d_1}, \ldots, x_c^{d_c})$.
\end{conj}

If true, Conjecture \ref{ConjectureEGH} would yield more accurate inequalities than Macaulay's Theorem for the Hilbert function and degree of subschemes of $\mathbb{P}^m$.
For instance, a typical application is the following statement: 
if $\Gamma \subseteq \mathbb{P}^m$ is a zero-dimensional subscheme cut out by  $m+1$ quadric hypersurfaces, then $\deg(\Gamma) \leq 2^m-2^{m-2}$.

Observe that, as in Macaulay's Theorem,  if the  ideal $L+(x_1^{d_1}, \ldots, x_c^{d_c})$  exists then it is unique, and it has the largest number of generators among all ideals with the same Hilbert function  containing a complete intersection of degrees $d_1\leq  \cdots \leq d_c$. 
The second conjecture, attributed to Charalambous and Evans in \cite{FR07}, predicts the same extremal behavior for all the syzygies:

\begin{conj}[Lex-plus-powers]\label{ConjectureLPP}
Let $I\subseteq S=\Bbbk[x_1,\ldots, x_n]$ be a homogeneous ideal containing a regular sequence of degrees $d_1\leq  \cdots \leq d_c$.
If there exists a lex ideal
 $L\subseteq S$  such that $I$ has the same Hilbert function as $L+(x_1^{d_1}, \ldots, x_c^{d_c})$, then
 $\beta_{i,j}^S(I) \leq \beta_{i,j}^S\big(L+(x_1^{d_1}, \ldots, x_c^{d_c})\big)$ for all $i,j\geq 0$.
\end{conj}

One can make these sharp upper bounds explicit by means of the formulas for Betti numbers of Borel-plus-powers ideals found in  \cite{M08}.
If Conjecture \ref{ConjectureLPP} were true, it would restrict considerably the possible shapes of free resolutions of ideals containing a given complete intersection.
It is worth noticing that, 
although the two statements are apparently independent of each other, 
Conjecture \ref{ConjectureEGH} is actually equivalent to the special case $i=0$ of Conjecture \ref{ConjectureLPP}, see e.g. \cite[Conjecture 2.9]{R04}.

Both conjectures are wide open. 
The classical  Clements-Linstr\"om Theorem \cite{CL69} settles Conjecture \ref{ConjectureEGH} for ideals $I$  already containing $(x_1^{d_1}, \ldots, x_c^{d_c})$.
Later,  the first author and Maclagan \cite{CM08} verified it for arbitrary regular sequences satisfying $d_i \geq \sum_{j=1}^{i-1} (d_j-1)+1$ for all $i \geq 3$.
See also \cite{A15,CCV14,C16,C12,G99,HP98,O02} for other special cases.
On the other hand, much less is known about Conjecture \ref{ConjectureLPP}, cf. \cite{F04,R04}.
The  case when $I$ already contains $(x_1^{d_1}, \ldots, x_c^{d_c})$ was settled only  recently and with complicated proofs, first in 
\cite{MPS08} when $d_1= \cdots = d_n = 2$ and $\ch(\Bbbk)=0$, and then in \cite{MM11} in general.
A main obstacle in both problems is the failure of Gr\"obner techniques to reduce to a monomial complete intersection while at the same time keeping track of homological data.

The main result of this paper settles Conjecture \ref{ConjectureLPP} in a large number of cases:

\begin{mainthm}
Assume that $\ch(\Bbbk)=0$.
The Lex-plus-powers Conjecture holds for all regular sequences whose degrees satisfy $d_i \geq \sum_{j=1}^{i-1} (d_j-1)+1$ for all $i\geq 3$.
\end{mainthm}

We point out that the Main  Theorem can be used to estimate the Betti table of any homogeneous ideal $I$ in characteristic 0, 
since if $I$ contains a regular sequence $\{f_1, \ldots, f_c\}$ then it contains another one $\{f_1, f_2, f'_3, \ldots, f'_c\}$ satisfying the hypothesis on  degrees.
In this way, the Main Theorem  provides general upper bounds for the Betti numbers that are sharper than those of the Bigatti-Hulett-Pardue Theorem,
but possibly worse than those predicted by the full strength of Conjecture \ref{ConjectureLPP},
cf. Section \ref{SectionExamples}.

The paper is organized as follows.
In Section \ref{SectionPreliminaries} we fix the notation and provide  background about lex-plus-powers ideals.
Section \ref{SectionProof} is dedicated to the proof of the main result. 
This is achieved by  inductively decomposing  ideals into smaller modules
and estimating their Betti numbers.
A crucial role in controlling the size of the smaller modules is played by Lemma \ref{LemmaHyperplaneSectionInequalityExtension},
which is inspired by Green's Hyperplane Restriction Theorem and only holds in characteristic 0.
Finally, in Section \ref{SectionExamples} we illustrate some explicit  bounds on Betti tables obtained from the Main Theorem.

\section{Preliminaries}\label{SectionPreliminaries}

In this section we fix the notation for the remainder of the paper and  give some definitions and preliminary results. 
We refer to \cite{E95} for background.

The symbol $\Bbbk$ denotes an arbitrary field;
in the main results we will need to assume  $\Char(\Bbbk)=0$.
 Let $R$ be a standard graded $\Bbbk$-algebra, we denote the unique maximal homogeneous ideal by $\mm_R$.
The length of an $R$-module is denoted by $\ell_R(M)$.
If $M$ is a graded $R$-module, $[M]_j$ is  the graded component of $M$ of degree $j$, and the Hilbert function of $M$ is the numerical function $\HF(M): \mathbb{Z}\rightarrow\mathbb{Z}$ such that $\HF(M;j) = \dim_\Bbbk [M]_j$.
For two graded modules $M,N$,
the expression $\HF(M) \preceq \HF(N)$ means that
 $\HF(M;j) \leq \HF(N;j)$ for all $j\in \mathbb{Z}$.
The symbol $M(j)$ denotes the graded module obtained from $M$ by twisting $j$ times, so that $[M(j)]_h = [M]_{j+h}$ for all $h \in \mathbb{Z}$.
We denote the graded Betti numbers of $M$ by 
$\beta_{i,j}^R(M) = \dim_\Bbbk [ \Tor_i^R(M,\Bbbk)]_j$ for all $i \in \mathbb{N}, j \in \mathbb{Z}$.

Throughout this paper, $S = \Bbbk[x_1, \ldots, x_n]$ denotes the polynomial ring in $n$ variables.
A \textbf{degree sequence} for $S$ is a vector $\bfd=(d_1, \ldots, d_n)$ where $d_i \in \mathbb{N}\cup\{\infty\}$ and $1 \leq d_1 \leq \cdots \leq d_n$.
We adopt the usual arithmetic conventions for $\mathbb{N}\cup\{\infty\}$ and  furthermore we set $x_i^\infty = 0$ for each $i$.
A \textbf{complete intersection} in $S$ is an ideal $\mathfrak{f} \subseteq S$ generated by a regular sequence of homogeneous forms $\{f_1, \ldots, f_c\}$;
we will always assume, without loss of generality, that $\deg(f_1) \leq \cdots \leq \deg(f_c)$.
The degree sequence of $\mathfrak{f}$ is the vector with $n$ entries $(\deg(f_1) , \ldots, \deg(f_c), \infty, \ldots, \infty)$.
We emphasize that the degree sequence of a complete intersection in $S$ has always length $n$, regardless of the actual codimension.
Observe that the Hilbert function of a complete intersection  is  uniquely determined by its degree sequence.

The variables of $S$ are ordered by $x_1 > x_2 > \cdots > x_n$ and we consider the lexicographic monomial order on $S$, denoted by $>_\lex$.
A monomial ideal $L\subseteq S$ is \textbf{lex} if for any two monomials $u,v\in S$ with $\deg(u) = \deg(v)$ and $u >_\lex v$ we have  $u\in L$ whenever $v\in L$.
By Macaulay's theorem, lex ideals in $S$ are in a one-to-one correspondence with Hilbert functions of ideals of $S$.
A monomial ideal $I\subseteq S$ is \textbf{$x_n$-stable} if for any monomial $u\in I$ divisible by  $x_n$ we have $\frac{x_i u}{x_n} \in I$ for all $i<n$.
Observe that a lex ideal is $x_n$-stable.

The next definitions play a central role in this paper. 
They extend  the definitions above to ideals containing prescribed pure powers of the variables of $S$.

\begin{definition}\label{DefinitionLPPandSPP}
Let $\bfd=(d_1, \ldots, d_n)$ be a degree sequence and  $\wp=(x_1^{d_1}, \ldots, x_n^{d_n})$. 
\begin{enumerate}
\item A \textbf{$\bfd$-lex-plus-powers} ideal or simply \textbf{$\bfd$-LPP} ideal is a monomial ideal of $S$ of the form $ L+\wp$ where $L$ is a lex ideal;
\item A \textbf{$\bfd$-stable-plus-powers} ideal or simply \textbf{$\bfd$-SPP} ideal is a monomial ideal of $S$ of the form $ I+\wp$ where $I$ is an $x_n$-stable ideal.
\end{enumerate}
\end{definition}
\noindent
We point out that, 
unlike   some sources in literature, in Definition \ref{DefinitionLPPandSPP} we do not require the generators of $\wp$  to be minimal generators of $L+\wp$ or $I+\wp$. 
Thus, a monomial ideal may be $\bfd$-SPP or $\bfd$-LPP for more than one degree sequence $\bfd$.

\begin{examples}
Let $n=3$ and $S= \Bbbk[x_1,x_2,x_3]$.
\begin{itemize}
\item  $I= (x_1^3 , x_1^2x_2,x_1^2x_3, x_2^3, x_1^2x_3^2, x_3^4)\subseteq S$
is $\bfd$-LPP if and only if $\bfd=(3,3,4)$.
\item  $I= (x_1^3 , x_1^2x_2, x_1^2x_3, x_1x_2^2, x_2^3, x_1^2x_3^2, x_3^4)\subseteq S$
is $\bfd$-LPP if and only if $\bfd=(d_1,d_2,4)$ with $3 \leq d_1 \leq d_2 \leq 4$.
\item  $I= (x_1^3 , x_1^2x_2, x_1^2x_3, x_2^3, x_1^2x_3^2, x_1x_2x_3^2, x_2^2x_3^2, x_1 x_3^3, x_2x_3^3, x_3^4)\subseteq S$
is $\bfd$-LPP if and only if $\bfd=(3,3,d_3)$ with $d_3 \geq 4$.

\end{itemize}
\end{examples}

The basic properties of LPP ideals  in the next proposition follow from their counterparts for lex ideals in $S$.

\begin{prop}\label{PropositionBasicPropertiesLPP}
Let $\bfd$ be a degree sequence.
\begin{enumerate}
\item If $L\subseteq S$ is a $\bfd$-LPP ideal, then it is a $\bfd$-SPP ideal;

\item if $L_1,L_2 \subseteq S$ are $\bfd$-LPP ideals  with $\HF(L_1) \preceq \HF(L_2)$, then $L_1 \subseteq L_2$.
\end{enumerate}

\end{prop}
\noindent
In particular, for each $\bfd$ there exists at most one $\bfd$-LPP ideal in $S$ with a given Hilbert function.
This motivates the following definition.

\begin{definition}
Let $I\subseteq S$ be a homogeneous ideal and $\bfd$ a degree sequence. 
If there exists a $\bfd$-LPP ideal in $S$ with the same Hilbert function as $I$, we denote it by $\LL^\bfd(I)$ and we refer to it as the  \textbf{$\bfd$-LPP ideal of $I$}.
\end{definition}

The existence of $\LL^\bfd(I)$ for every ideal $I$ containing a complete intersection of degree sequence $\bfd$ is precisely the content of Conjecture \ref{ConjectureEGH}.
It is guaranteed for those $\bfd$ that increase ``quickly enough''.

\begin{prop}
[{\cite[Theorem 2]{CM08}}]\label{PropositionEGHCavigliaMaclagan}
Let $I\subseteq S$ be a homogeneous ideal containing a regular sequence of degrees $\bfd = (d_1, \ldots, d_n)$.
Assume that $d_i \geq \sum_{j= 1}^{i-1} (d_j -1)+1 $ for all $i\geq 3$, then $\LL^\bfd(I)$ exists.
\end{prop}

While Conjectures \ref{ConjectureEGH} and \ref{ConjectureLPP} are of interest for complete intersections of any codimension, in practice one can reduce to considering $\mm_S$-primary ideals, i.e. degree sequences $\bfd$ with $d_n < \infty$.

\begin{prop}[{\cite[Theorem 4.1]{CK14}}]\label{PropositionReductionToArtinianCavigliaKummini}
Let $\{f_1, \ldots, f_c\}\subseteq S$ be a regular sequence. 
If the Eisenbud-Green-Harris and Lex-plus-powers Conjectures hold for the image of $\{f_1, \ldots, f_c\}$ modulo $n-c$ general linear forms,  then they hold for $\{f_1, \ldots, f_c\}$.
\end{prop}

In several arguments to follow, we will proceed by induction on the number $n$ of variables of $S$, and it is convenient to adopt the subsequent notation. 
Let $\ovS= \Bbbk[x_1,\ldots,x_{n-1}]$ and consider $S$ as the extension $\ovS[x_n]$.
We let $\bfd = (d_1, \ldots, d_n)$ be a degree sequence with $n$ entries, and denote by $\obfd = (d_1, \ldots, d_{n-1})$ the partial sequence of the first $n-1$ entries.
We let $\wp = (x_1^{d_1}, \ldots, x_n^{d_n}) $ be the monomial complete intersection in $S$ corresponding to $\bfd$, and denote by
$\opp = (x_1^{d_1}, \ldots, x_n^{d_{n-1}}) $  the monomial complete intersection in $\ovS$ corresponding to $\obfd$. 

A monomial ideal $I \subseteq S$ can be decomposed uniquely as 
$$
I = I_0 \oplus I_1 x_n \oplus I_2 x_n^2 \oplus \cdots
$$
where $I_i \subseteq \ovS$ is a monomial ideal and  $I_{i-1} \subseteq I_{i}$ for all $i > 0$.
Unless stated otherwise, in this paper the notation $I_i$  always refers to such decomposition of monomial ideals.
We list some basic properties of monomial ideals of $S$ in terms of this decomposition;
their proofs follow immediately from the definitions.

\begin{prop}\label{PropositionBasicPropertiesDecomposition}
Let $I\subseteq S$ be a monomial ideal and $\bfd$ a degree sequence.
\begin{enumerate}
\item $I$ is $\bfd$-SPP if and only if  $\wp\subseteq I$ and $\mm_\ovS I_i \subseteq I_{i-1}$ for all  $0<i<d_n $.

\item If $I$ is a  $\bfd$-SPP ideal then for all  $0\leq i<d_n $ we have

$$ \frac{(I: x_n^i)+ (x_n)}{(x_n)} \cong I_i  \subseteq \ovS.$$

\item If  $I$ is a  $\bfd$-LPP ideal, then $I_i\subseteq \ovS$ is a $\obfd$-LPP ideal for every $i$.

\end{enumerate}
\end{prop}

We conclude this section by collecting some basic facts about linkage.
Let $I\subsetneq S$ be an unmixed homogeneous ideal, and let $\ff \subsetneq I$ be a complete intersection of the same codimension.
The ideal $ J = \mathfrak{f} : I$ is said to be \textbf{directly linked} to $I$ via $\ff$;
$J$ is  an unmixed homogeneous ideal of the same codimension as $I$,  and $\mathfrak{f} \subsetneq J$.
This operation is a duality in the sense that $ I =   \mathfrak{f} : J$.
Now let $\mathfrak{f}= (f_1, \ldots, f_n) \subseteq S$ be an $\mm_S$-primary complete intersection of degree sequence $\bfd=(d_1, \ldots, d_n)$, with $d_n<\infty$.
Then $S/\ff$ is Artinian with socle degree $s=\sum_{i=1}^n d_i -n$, and the Hilbert functions of the linked ideals $I$ and $J = \mathfrak{f} :I$ satisfy
\begin{equation}\label{EqHilbertFunctionLinkage}
\HF\left(\frac{S}{I}; j\right) + \HF\left(\frac{S}{J};s-j\right) = \HF\left(\frac{S}{\ff};j\right)
\end{equation}
for all $j = 0, \ldots, s$, cf. \cite[5.2.19]{M12}.
In particular, the Hilbert function of a direct link $J = \mathfrak{f}:I$ of an $\mm_S$-primary ideal $I$ depends only on $\HF(I)$ and on the degree sequence of $\mathfrak{f}$.

\section{Proof of the main result}\label{SectionProof}

We begin this section by investigating the behavior of lex-plus-powers ideals under linkage.
We show that taking a direct link of an LPP ideal via the regular sequence of pure powers yields another LPP ideal. 
This fact is already known;  it is proved in \cite{RS08}.
Here we include a shorter alternative proof,
which relies on the following recursive characterization of LPP ideals.

\begin{remark}[{\cite[Proof of Theorem 3.3, Lemma 3.7, Lemma 3.8]{CK13}}]\label{RemarkCharacterizationLPP}
An ideal $L \subseteq S$ is $\bfd$-LPP if and only if the following conditions hold:

\begin{itemize}
\item[(i)] $L$ is $\bfd$-SPP;
\item[(ii)] $L_i \subseteq \ovS$ is $\obfd$-LPP for all $i$;
\item[(iii)] if $I\subseteq S$ is another $\bfd$-SPP ideal with $\HF(I)=\HF(L)$, then   for  all $i,p \geq 0$
$$\sum_{j=0}^i\HF\left(I_j; p-j\right) \geq  \sum_{j=0}^i\HF\left(L_j; p-j\right).$$
\end{itemize}
\end{remark}

\begin{prop}\label{PropositionLinkageMonomial}
Let $\bfd$ be a degree sequence with $d_n < \infty$. 
If  $I$ is a monomial ideal with $\wp  \subsetneq I \subsetneq S$ and $J = \wp :I$ then
\begin{enumerate}
\item $J_i = \opp : I_{d_n-i-1} \subseteq \ovS$ for each  $i = 0, \ldots, d_n-1$;

\item  $I$ is $\bfd$-SPP if and only if $J$ is $\bfd$-SPP;

\item  $I$ is  $\bfd$-LPP  if and only if $J$ is $\bfd$-LPP.

\end{enumerate}

\end{prop}

\begin{proof}
Recall that $J$ is monomial with $\wp \subsetneq J \subsetneq S$, and we have decompositions
\begin{eqnarray*}
I  &=& I_0 \oplus I_1 x_n \oplus \cdots I_{d_n-1} x_n^{d_n -1} \oplus \ovS x_n^{d_n} \oplus \ovS x_n^{d_n +1} \oplus \cdots,\\
J  &=& J_0 \oplus J_1 x_n \oplus \cdots J_{d_n-1} x_n^{d_n -1} \oplus \ovS x_n^{d_n} \oplus \ovS x_n^{d_n +1} \oplus \cdots.
\end{eqnarray*}

(1) For each $i = 0, \ldots, d_n-1$ we have  $\opp: I_i \subseteq J_{d_n -1-i}$,
because 
\begin{eqnarray*}
(\opp: I_i) x_n^{d_n-1-i}(I_j x_n^j) \subseteq (\opp: I_i) x_n^{d_n-1-i}(I_i x_n^j) \subseteq\opp S\subseteq \wp &\mbox{ for } j \leq i,\\
(\opp: I_i) x_n^{d_n-1-i}(I_j x_n^j) \subseteq (x_n^{d_n}) \subseteq \wp  &\mbox{ for }  j > i.
\end{eqnarray*}
 On the other hand, 
applying \eqref{EqHilbertFunctionLinkage} to $I$ and to each $I_i$, 
we see that $\opp: I_i $ and $ J_{d_n -1-i}$ must have the same Hilbert function for every $i = 0, \ldots, d_n-1$, 
therefore equality must hold.

(2) 
If $I$ is $\bfd$-SPP, for each $i = 1,\ldots,d_n-1$ we have $\mm_\ovS I_{d_n - i} \subseteq I_{d_n-i-1}$ by Proposition \ref{PropositionBasicPropertiesDecomposition}. 
Passing to  links via $\opp$ and using part (1) we obtain 
$$
J_{i} = \opp :  I_{d_n-i-1}  \subseteq  \opp :  \mm_\ovS I_{d_n - i} =  (\opp : I_{d_n - i}) : \mm_\ovS = J_{i-1} : \mm_\ovS 
 $$ 
 and hence $\mm_\ovS J_i \subseteq J_{i-1} $. 
 Thus, $J$ is a $\bfd$-SPP ideal as desired.
 
(3) 
Let $I$ be a $\bfd$-LPP ideal.
We proceed by induction on $n$, the case $n=0$ being trivial.
In order to prove that $J$ is $\bfd$-LPP,  we use Remark \ref{RemarkCharacterizationLPP}:
by induction, (i) follows from Proposition \ref{PropositionBasicPropertiesLPP} and part (2), whereas (ii) follows from Proposition \ref{PropositionBasicPropertiesDecomposition} and part (1).
Assume by contradiction that (iii) fails, that is, 
there exist  another $\bfd$-SPP ideal  $J'\subseteq S$ with $\HF(J') = \HF(J)$ and 
values $i,p\geq 0 $ such that
$\sum_{j=0}^i\HF\left(J'_j; p-j\right) <  \sum_{j=0}^i\HF\left(J_j; p-j\right).$
Taking direct links of each $J_j, J'_j$ via $\opp$, equation \eqref{EqHilbertFunctionLinkage} gives
$$\sum_{j=0}^i\HF\left(\opp : J'_j; \overline{s} - p+j\right) >  \sum_{j=0}^i\HF\left(\opp : J_j; \overline{s}-p+j\right)$$
where $\overline{s} = \sum_{j=1}^{n-1} (d_j -1)$ is the socle degree of $\ovS/\opp$.

The ideal $I' = \wp : J' \subseteq S$ is  $\bfd$-SPP  with $\HF(I')= \HF(I)$ and $I'_j = \opp : J'_{d_n-j-1}$ for each $j=0,\ldots, d_n-1$, by parts (1) and (2).
The previous inequality becomes
$
\sum_{j=0}^i\HF(I'_{d_n-j-1}; \overline{s} - p+j) >  \sum_{j=0}^i\HF\left(I_{d_n-_j-1}; \overline{s}-p+j\right).
$
Setting  $q = \overline{s}-p+d_n-1$ and reindexing, we obtain
\begin{equation}\label{EqInequalityLastComponentsVector}
\sum_{j=d_n-i-1}^{d_n-1}\HF\left(I'_j; q-j\right) >  \sum_{j=d_n-i-1}^{d_n-1}\HF\left(I_j; q-j\right).
\end{equation}
Observe that 
$
\sum_{j=0}^{q} \HF(I'_j; q-j)= \HF(I';q)  =\HF(I;q) = \sum_{j=0}^q \HF(I_j; q-j).
$
Furthermore, since $I_j = I'_j = \ovS$ for $j\geq d_n$, we actually have 
$
\sum_{j=0}^{d_n} \HF(I'_j; q-j)=  \sum_{j=0}^{d_n} \HF(I_j; q-j),
$ 
so by \eqref{EqInequalityLastComponentsVector} we deduce
$$
\sum_{j=0}^{d_n-i-2}\HF\left(I'_j; q-j\right) <  \sum_{j=0}^{d_n-i-2}\HF\left(I_j; q-j\right).
$$
However, this  contradicts the fact that $I$ is a $\bfd$-LPP ideal, by Remark \ref{RemarkCharacterizationLPP}.
Thus, $J$ is a $\bfd$-LPP ideal, and the proof is concluded.
\end{proof}

Next we prove an inequality of Hilbert functions between the general hyperplane sections of certain ideals and those of their LPP ideals, in the spirit of Green's Hyperplane Restriction Theorem \cite{G89} and its generalizations \cite{HP98,G99}.
Our statement only holds if the ground field has characteristic 0.

\begin{lemma}\label{LemmaHyperplaneSectionInequalityExtension}
Assume  $\ch(\Bbbk)=0$. 
Let $\{f_1, \ldots, f_{n-1}\} \subseteq S$ be a regular sequence with  degree sequence $\bfd$, where  $d_n = \infty$,
and $I \subseteq S$  a homogeneous ideal containing $\{f_1, \ldots, f_{n-1} \}$.
Let  $\ell \in [S]_1$ be a general linear form and 
denote by $\overline{f}_i$ the image of $f_i$ in  $ S/(\ell)\cong \ovS$.
If Conjecture \ref{ConjectureEGH} holds for the regular sequence $\{\overline{f}_1, \ldots, \overline{f}_{n-1} \} $
then for all $i \geq 1$ we have 
 $$ \HF\big(I + (\ell^i)\big) \succeq \HF\big(\LL^\bfd(I)+(x_n^i)\big).$$
\end{lemma}
\begin{proof}
Up to applying a general change of coordinates, 
we may already assume that the coordinates $x_1, \ldots, x_n$ are general and that $\ell = x_n$.
Note that the assumptions are preserved by changes of coordinates.

Let $K = \initial_{\omega}(I)$ be the  initial ideal with respect to the weight $\omega=(1,1,\ldots1, 0)\in \mathbb{N}^n$.
Then $K$ is bihomogeneous by letting  $\bideg(x_i)=(1,0)$ for $i< n$ and $\bideg(x_n)=(0,1)$, and we may decompose
$
K = K_0 \oplus K_1 x_n \oplus K_2 x_n^2 \oplus \cdots 
$
where $K_j \subseteq \ovS$ is a homogeneous ideal containing $\{\overline{f}_1, \ldots, \overline{f}_{n-1} \} $.
By construction we have $\initial_{\omega}(I+ (x_n^i)) = K + (x_n^i)$  for all $i\geq 1$.
Since the coordinates are general and $\ch(\Bbbk)=0$ we have  $\mm_\ovS K_j \subseteq K_{j-1}$ for all $j\geq 1$ by 
\cite[Proposition 2.17]{G98}.

By assumption the ideal  $H_j = \LL^\obfd(K_j) \subseteq \ovS$ is well defined for every $j$.
Consider the graded $\ovS$-module
$
H = H_0 \oplus H_1 x_n \oplus H_2 x_n^2 \oplus \cdots,
$
by construction we have $\HF(H) = \HF(K)$ and in fact $\HF(H + (x_n^i)) = \HF(K+ (x_n^i))$ for all $i \geq 1$.
For each $j \geq 1$ the containment 
 $K_{j-1} \subseteq K_{j}$ implies $\HF(H_{j-1}) = \HF(K_{j-1}) \preceq \HF(K_j) = \HF(H_j) $ and hence  $H_{j-1} \subseteq H_{j}$ by Proposition \ref{PropositionBasicPropertiesLPP}, 
 so $H$ is an ideal of $S$.
Moreover, the containment
 $ \mm_\ovS  K_{j} \subseteq K_{j-1}$ 
 implies $\HF(\mm_\ovS K_{j}) \preceq \HF(K_{j-1}) = \HF(H_{j-1})$.
Since Conjecture \ref{ConjectureEGH} is true for the regular sequence $\{\overline{f}_1, \ldots, \overline{f}_{n-1} \} $,
 the inequality $  \HF(\mm_\ovS H_{j})  \preceq \HF(\mm_\ovS K_{j})$ holds; 
in fact, this is a well-known consequence of  Conjecture \ref{ConjectureEGH}, see e.g. \cite[Lemma 2.4]{CS16}.
We deduce  $ \HF(\mm_\ovS H_{j}) \preceq  \HF(H_{j-1})$, and by Proposition \ref{PropositionBasicPropertiesLPP} it follows that   $\mm_\ovS H_j \subseteq H_{j-1}$, thus  $ H$ is a $\bfd$-SPP ideal.

To summarize, we have $\HF(I + (x_n^i)) = \HF(K + (x_n^i)) = \HF(H + (x_n^i))$ for all $i\geq 1$, 
and $H$ is a $\bfd$-SPP ideal of $S$ with $\HF(I) = \HF(H)$, so in particular $\LL^\bfd(I) = \LL^\bfd(H)$.
Now we may apply \cite[Theorem 3.9]{CK13}: in the language of that paper, 
the ring $\frac{S}{\wp}  $ has the embedding $ \frac{\mathcal{I}}{\wp} \mapsto \frac{\LL^\bfd(\mathcal{I})}{\wp}$,
and this is precisely the embedding produced by  \cite[Proof of Theorem 3.3]{CK13} starting from the embedding of $\frac{\ovS}{\opp}$ given by $ \frac{\mathcal{J}}{\opp} \mapsto \frac{\LL^\obfd(\mathcal{J})}{\opp}$.
We obtain that $\HF(H + (x_n^i)) \succeq \HF( \LL^\bfd(H) + (x_n^i))$, 
and the desired conclusion follows.
\end{proof}

\begin{example}
Let $\Bbbk$ be a field with $\ch(\Bbbk)=p>0$ and $S = \Bbbk[x_1,x_2,x_3]$.
Let  $I = (x_1^{2p},x_1^px_2^p,x_2^{2p},x_1^px_3^p, x_2^px_3^p)\subseteq S$;
choosing degree sequence $\bfd=(2p,2p,\infty)$, the lex-plus-powers ideal is $\LL^\bfd(I)=(x_1^{2p}, x_1^{2p-1}x_2, x_1^{2p-1}x_3, x_1^{2p-2}x_2^2,x_2^{2p} ) + K$ 
for some monomial ideal $K$ generated in degrees $\geq 2p+1$.
For any $\ell \in [S]_1$ with non-zero coefficient in $x_3$
the image of $[I]_{2p}$ modulo $(\ell)$ is a 3-dimensional vector space,
whereas the image of $[\LL^\bfd(I)]_{2p}$ modulo $(x_3)$ is 4-dimensional.
Thus
 $ \HF(I + (\ell); 2p) < \HF(\LL^\bfd(I)+(x_3);2p)$ and
 the conclusion of Lemma \ref{LemmaHyperplaneSectionInequalityExtension} is false.
\end{example}

We recall  the following well-known fact about graded Betti numbers.

\begin{remark}\label{RemarkShortExactSequenceBetti}
Given a  standard graded $\Bbbk$-algebra $R$ and a short exact sequence of finitely generated graded $R$-modules
$
0 \rightarrow M_1 \rightarrow M_2 \rightarrow M_3 \rightarrow 0,
$
we have $\beta_{i,j}^R(M_2) \leq \beta_{i,j}^R(M_1)+ \beta_{i,j}^R(M_3)$ for all $i,j$.
\end{remark}

We introduce a notation for the ``underlying graded vector space'' of a graded module of finite length; 
these objects will be useful when estimating Betti numbers over $\ovS$ in the proof of the main theorem.

\begin{notation}\label{NotationUnderlyingVectorSpace}
Let $R$ be a standard graded $\Bbbk$-algebra and $M$ a graded $R$-module with $\ell_R(M)<\infty$.
We consider the graded  $R$-module 
$$
V(M)= \bigoplus_{j\in\mathbb{Z}}\Bbbk(-j)^{\HF(M;j)}.
$$ 
It has the same Hilbert function as $M$ and it is annihilated by $\mm_R$.
By induction on $\ell_R(M)$ and Remark \ref{RemarkShortExactSequenceBetti} we see that $\beta_{i,j}^R(M) \leq \beta_{i,j}^R(V(M))  $ for all $i,j$.
Furthermore, the Betti numbers of $V(M)$ are uniquely determined by $\HF(M)$.
\end{notation}

\begin{lemma}\label{LemmaBettiHyperplaneSectionStable}
Let $I \subseteq S$ be a $\bfd$-SPP ideal, where $\bfd$ is a degree sequence with  $d_n < \infty$.
Then for all $i,j\geq 0$ we have
$$
\beta_{i,j}^\ovS\left(\frac{I}{x_nI}\right)  = \beta_{i,j}^\ovS(I_0) +\beta_{i,j}^\ovS\left(V\left(\bigoplus_{h=1}^{d_n-1} \frac{I_h}{I_{h-1}}(-h)\right)\right)  + \beta_{i,j}^\ovS\left(\frac{\ovS}{I_{d_n-1}}(-d_n)\right).
$$
\end{lemma}
\begin{proof}
This follows immediately from the decomposition of graded $\ovS$-modules
$ I = I_0 \oplus I_1 x_n \oplus \cdots \oplus I_{d_n-1} x_n^{d_n-1} \oplus \ovS x_n^{d_n} \oplus \cdots$
and the fact that $\frac{I_h}{I_{h-1}}$ is already annihilated by $\mm_\ovS$  for $1 \leq i \leq d_n-1$ by Definition \ref{DefinitionLPPandSPP}.
\end{proof}

We are now ready to prove the main result of this paper.

\begin{thm}\label{TheoremMainInBodyPaper}
Let $I\subseteq S=\Bbbk[x_1,\ldots, x_n]$ be a homogeneous ideal containing a complete intersection of degree sequence $\bfd$, such that $d_i \geq \sum_{j=1}^{i-1}(d_j-1)+1$ for all $i\geq 3$.
Assume $\ch(\Bbbk)=0$.
Then $\beta_{i,j}^S(I) \leq \beta_{i,j}^S\big(\LL^\bfd(I)\big)$ for all $i,j\geq 0$.
\end{thm}

\begin{proof}
By Proposition \ref{PropositionReductionToArtinianCavigliaKummini} we may assume that $d_n < \infty$. 
From  Proposition \ref{PropositionEGHCavigliaMaclagan} we know that Conjecture \ref{ConjectureEGH} holds for $\bfd$ in $S$ and $\obfd$ in $\ovS$.
We prove the theorem by induction on $n$, the cases $n=0,1$ being trivial.

Let $\mathfrak{f}=(f_1, \ldots, f_n)\subseteq I$ be the given complete intersection, with $\deg(f_i)=d_i$.
We consider the $S$-ideals $J = \mathfrak{f}: I$,  $L= \LL^\bfd(I) $, and $K = \LL^\bfd(J)$.
Notice that $L$ and $K$ are well-defined, and they are directly linked via $\wp$ by Proposition \ref{PropositionLinkageMonomial}.

By changing coordinates, we may assume that $x_n$ is a general linear form.
Then the ideal $(f_1, \ldots, f_{n-1}, x_n)$ is an $\mm_S$-primary complete intersection with socle degree $\sum_{j=1}^{n-1} (d_j -1)$.
By assumption $d_n > \sum_{j=1}^{n-1} (d_j -1)$ and thus $f_n \in (f_1, \ldots, f_{n-1}, x_n)$,
therefore, up to replacing the form $f_n$, we may assume that $f_n = x_n g$ for some  form $g \in [S]_{d_n-1}$.

Since $x_n$ is a linear non-zerodivisor in $S$, for all $i,j\geq 0$ we have 
\begin{equation*}
\beta_{ij}^S(I) = \beta_{ij}^{\ovS}\left(\frac{I}{x_n I}\right)\qquad \mbox{and}\qquad \beta_{ij}^S(L) = \beta_{ij}^{\ovS}\left(\frac{L}{x_n L}\right)
\end{equation*}
thus it will suffice to show 
$ \beta_{ij}^{\ovS}\left(\frac{I}{x_n I}\right)\leq \beta_{ij}^{\ovS}\left(\frac{L}{x_n L}\right)$
for every $i, j \geq 0$.

There is a short exact sequence of graded $\ovS$-modules
\begin{equation}\label{EqShortExactSequence}
0 \longrightarrow 	 \frac{x_n(I:x_n)}{x_nI} \longrightarrow \frac{I }{x_n I} \longrightarrow \frac{I+(x_n)}{(x_n)}  \longrightarrow 0.
\end{equation}

The  $\ovS$-ideal $  \overline{I}  =  \frac{I+(x_n)}{(x_n)} $ contains $\frac{(f_1, \ldots, f_{n-1},x_n)}{(x_n)}$, which is a complete intersection in  $\ovS$ of degree sequence $\obfd$.
By induction, $\beta_{i,j}^{\ovS}(\overline{I}) \leq \beta_{i,j}^\ovS (\LL^{\overline{\bfd}}(\overline{I}))$ for all $i,j$.
Since both $S$-ideals $I + (x_n)$ and $L+(x_n)$ contain an $\mm_S$-primary complete intersection of socle degree $\sum_{j=1}^{n-1}(d_j-1)$, we have  $[I + (x_n)]_d =[L + (x_n)]_d =[S]_d$ for all $d \geq d_n$.
On the other hand, when $d < d_n $  we have  $[L + (x_n) ]_d =  [ \LL^{\mathbf{e}}(I) + (x_n)]_d$ for $\mathbf{e}= (d_1, \ldots, d_{n-1}, \infty)$, and
applying Lemma \ref{LemmaHyperplaneSectionInequalityExtension} we obtain 
$ \HF(I + (x_n); d) \geq \HF( \LL^{\mathbf{e}}(I) + (x_n); d) = \HF(L + (x_n); d)$ for all $d < d_n$.
Combining the two cases $d < d_n$ and $d \geq d_n$ we get
$\HF(I + (x_n))\succeq  \HF(L + (x_n))$,
equivalently, 
$\HF(\overline{I}) \succeq \HF({L_0})$.
This inequality implies the containment of $\obfd$-LPP ideals $L_0 \subseteq \LL^{\overline{\bfd}}(\overline{I})$ by Proposition \ref{PropositionBasicPropertiesLPP}, since $L_0$ is a $\obfd$-LPP ideal of $\ovS$ 
by Proposition \ref{PropositionBasicPropertiesDecomposition} and $\HF({\overline{I}}) = \HF({\LL^{\obfd}(\overline{I})})$ by definition.
By Remark \ref{RemarkShortExactSequenceBetti} and Notation \ref{NotationUnderlyingVectorSpace}   we deduce
\begin{equation}\label{EqEstimateBettiHyperplaneSection}
\beta_{i,j}^{\ovS}\left(\overline{I}\right) \leq \beta_{i,j}^\ovS(L_0) + \beta_{i,j}^\ovS \left(V\left( \frac{\LL^{\overline{\bfd}}\left(\overline{I}\right)}{L_0}\right)\right).
\end{equation}

Next, we consider the first term in the exact sequence \eqref{EqShortExactSequence}.
We have the inclusion and isomorphsims of graded $\ovS$-modules
$$
\frac{S}{I:g}(-d_n) \cong \frac{(g)}{g(I:g)}(-1) \cong \frac{I+(g)}{I}(-1) \subseteq \frac{I:x_n}{I}(-1) \cong \frac{x_n(I:x_n)}{x_n I}
$$
and hence, as in the previous paragraph, we deduce
\begin{equation}\label{EqEstimateBettiFristTermSES}
\beta_{i,j}^{\ovS}\left(\frac{x_n(I:x_n)}{x_nI}\right) \leq  \beta_{i,j}^{\ovS} \left(V \left(\frac{I:x_n}{I+(g)}(-1)\right)\right)+ \beta_{i,j}^{\ovS}\left(\frac{S}{I:g}(-d_n)\right).
\end{equation}

Applying the same argument as above to the  $S$-ideals $J + (x_n)$ and $K+(x_n)$, we have  $[J + (x_n)]_d =[K + (x_n)]_d =[S]_d$ for  $d \geq d_n$ and
  $[K + (x_n) ]_d =  [ \LL^{\mathbf{e}}(J) + (x_n)]_d$ for  $d < d_n $,
  thus   by Lemma \ref{LemmaHyperplaneSectionInequalityExtension} we get
$$
\HF(J + (x_n)) \succeq \HF(K + (x_n)).
$$
Taking direct links of the ideals in this inequality via $\mathfrak{f}$ and $\wp$ respectively,
\eqref{EqHilbertFunctionLinkage} yields the inequality of Hilbert functions
$$
\HF\big(\mathfrak{f}:(J + (x_n))\big) \preceq \HF\big(\wp:(K + (x_n))\big).
$$
Notice that $\mathfrak{f}:(J + (x_n))=I \cap (g) = g(I:g)$ and $\wp:(K + (x_n)) = x_n^{d_n-1}(L:x_n^{d_n-1})$,
therefore, since  $\deg(g)=\deg(x_n^{d_n-1})$,
the inequality of Hilbert functions becomes
\begin{equation}\label{EqInequalityHilbertFunctionProof}
\HF(I:g) \preceq \HF(L:x_n^{d_n-1}).
\end{equation}

Let $I_\star\subseteq \ovS$ denote the image of $I:g$ in $\ovS \cong S/(x_n)$,
so  $I:g = I_\star S + (x_n)\subseteq S$.
Since  $I_\star$ contains a complete intersection of degree sequence $\obfd$,
 we may consider $\LL^\obfd(I_\star) \subseteq \ovS$.
By induction $\beta_{i,j}^\ovS(I_\star) \leq \beta_{i,j}^\ovS(\LL^\obfd(I_\star))$ for all $i,j$, or,
equivalently, 
 \begin{equation}\label{EqEstimateBettiLastComponentI}
  \beta_{i,j}^\ovS\left(\frac{S}{I_\star S + (x_n)} \right) \leq \beta_{i,j}^\ovS\left(\frac{S}{\LL^\obfd(I_\star) S + (x_n)}\right).
 \end{equation}

Observe that $L:x_n^{d_n-1}=L_{d_n-1}S + (x_n)\subseteq S$.
The variable  $x_n$ is a  non-zerodivisor modulo the extended ideals $I_\star S$ and $L_{d_n-1}S$;
 combining with  \eqref{EqInequalityHilbertFunctionProof}, we deduce the inequality $\HF({I_\star}) \preceq \HF(L_{d_n-1})$.
Since $\HF({I_\star}) = \HF(\LL^\obfd(I_\star))$ by definition, 
we have $ \HF(\LL^\obfd(I_\star)) \preceq \HF(L_{d_n-1})$,
and since both   $\LL^\obfd(I_\star) $ and $L_{d_n-1}$ are  $\obfd$-LPP ideals of $\ovS$, 
we conclude by Proposition \ref{PropositionBasicPropertiesLPP} that $\LL^\obfd(I_\star) \subseteq L_{d_n-1}\subseteq \ovS$ and therefore $\LL^\obfd(I_\star)S+(x_n) \subseteq L_{d_n-1}S + (x_n)\subseteq S$.
Applying Remark \ref{RemarkShortExactSequenceBetti}  and Notation \ref{NotationUnderlyingVectorSpace} 
to the short exact sequence of graded $\ovS$-modules
$$
0 \longrightarrow 	 \frac{L_{d_n-1}S + (x_n)}{\LL^\obfd(I_\star)S+(x_n)} \longrightarrow \frac{S }{\LL^\obfd(I_\star)S+(x_n)} \longrightarrow \frac{S}{L_{d_n-1}S + (x_n)}  \longrightarrow 0.
$$
and combining with \eqref{EqEstimateBettiLastComponentI} we obtain
 \begin{equation}\label{EqEstimateBettiIcolon}
  \beta_{i,j}^\ovS\left(\frac{S}{I: g} \right) \leq\beta_{i,j}^\ovS\left(  V \left(  \frac{L_{d_n-1} S+ (x_n)}{\LL^\obfd(I_\star)S+(x_n)} \right)  \right)+  \beta_{i,j}^\ovS\left( \frac{S}{L_{d_n-1}S + (x_n)}   \right).
 \end{equation}

Finally, 
we combine all the information obtained thus far to estimate the Betti numbers of $\frac{I}{x_nI}$.
By \eqref{EqShortExactSequence}, \eqref{EqEstimateBettiHyperplaneSection}, \eqref{EqEstimateBettiFristTermSES}, and \eqref{EqEstimateBettiIcolon} we have for all $i,j\geq 0$
\begin{eqnarray*}
 \beta_{i,j}^\ovS \left( \frac{I }{x_n I}\right) & \leq & \beta_{i,j}^\ovS(L_0) + \beta_{i,j}^\ovS \left(  V\left( \frac{\LL^{\overline{\bfd}}\left(\overline{I}\right)}{L_0}\right) \right) +\beta_{i,j}^{\ovS} \left(V \left(\frac{I:x_n}{I+(g)}(-1)\right)\right)\\
  &+&\beta_{i,j}^\ovS\left(  V \left(  \frac{L_{d_n-1}S + (x_n)}{\LL^\obfd(I_\star)S+(x_n)}(-d_n) \right)  \right)+  \beta_{i,j}^\ovS\left( \frac{S}{L_{d_n-1} S+ (x_n)}  (-d_n) \right)\\
& =& \beta_{i,j}^\ovS(L_0) + \beta_{i,j}^\ovS(W)+ \beta_{i,j}^\ovS\left( \frac{\ovS}{L_{d_n-1}} (-d_n)  \right) =: B_{i,j}
\end{eqnarray*}
where $W$ is the graded vector space
$$
W= V\left( \frac{\LL^{\overline{\bfd}}\left(\overline{I}\right)}{L_0} \oplus \frac{I:x_n}{I+(g)}(-1) \oplus  \frac{L_{d_n-1} S+ (x_n)}{\LL^\obfd(I_\star)S+(x_n)}(-d_n) \right).
$$
We claim that the sum $B_{i,j}$ is precisely the Betti number $\beta_{i,j}^\ovS(\frac{L}{x_nL})$.
This follows from Lemma \ref{LemmaBettiHyperplaneSectionStable} once we verify that the graded vector spaces $W$ and $\bigoplus_{h=1}^{d_n-1}\frac{L_h}{L_{h-1}}(-h)$ have the same Hilbert function.
However, this is true because in each step \eqref{EqShortExactSequence}, \eqref{EqEstimateBettiHyperplaneSection}, \eqref{EqEstimateBettiFristTermSES}, and \eqref{EqEstimateBettiIcolon}  we replaced a graded $\ovS$-module by another one with the same Hilbert function,
and the Hilbert functions of $\frac{I}{x_nI }$ and $\frac{L}{x_nL }$ coincide.

We have proved that $\beta_{i,j}^\ovS(\frac{I}{x_nI})\leq\beta_{i,j}^\ovS(\frac{L}{x_nL})$ for all $i,j\geq 0$, and, as already observed, this concludes the proof.
\end{proof}

\section{Examples}\label{SectionExamples}

We conclude this paper by illustrating the upper bounds on Betti tables obtained from Theorem \ref{TheoremMainInBodyPaper} in some specific examples.

Throughout the section we assume  $\ch(\Bbbk)=0$.
We adopt the usual ``Macaulay notation'' for writing Betti tables, placing the  number $\beta_{i,j}$ in column $i$ and row $j-i$.
Graded Betti numbers of lex ideals are determined by the well-known Eliahou-Kervaire  formulas \cite{EK90},
whereas  graded Betti numbers of  LPP ideals can be calculated via  the formulas of \cite{M08}.

\begin{example}\label{ExamplePointsCodim3}
Let $\Omega \subseteq \mathbb{P}^3$ be a 0-dimensional complete intersection of degrees $(4,4,8)$. 
Let $\Gamma \subseteq \Omega $ be a closed subscheme with Hilbert function 
$$\HF(A /I_\Gamma) = (1, 4, 10, 20, 32, 44, 56, 68, 79, 88, 94, 96, 96, 96, \ldots)$$
where $A= \Bbbk[x_0, x_1, x_2, x_3]$ is the homogeneous coordinate ring of $\mathbb{P}^3$.
By going modulo a general linear form in $A/I_\Gamma$, we reduce to considering
 Artinian algebras $R =S/I$ where  $S= \mathbb{\Bbbk}[x_1,x_2,x_3]$ with $\HF(R) = (1, 3, 6, 10, 12, 12, 12, 12, 11, 9, 6, 2)$ and  $I$ contains a regular sequence of degrees $(4,4,8)$.
Computing the lex ideal and the $(4,4,8)$-LPP ideal of $I$ in $S$,  we obtain
the upper bounds for the Betti table of $I_\Gamma$ provided  respectively by the Bigatti-Hulett-Pardue Theorem and  Theorem \ref{TheoremMainInBodyPaper}:

$$
\begin{array}{r|ccc}
& 0 & 1 & 2 \\
\hline
4 &  3 & 3 &1 \\
5 &  3& 5 &2 \\
6 &  2 & 3 &1 \\
7 &  1 & 2 &1 \\
8 &  2 & 4 &2 \\
9 &  3 & 5 &2 \\
10 & 3 & 6 &3 \\
11 &  4 & 8 &4 \\
12 & 2 & 4 &2 
\end{array}
\hspace*{3cm}
\begin{array}{r|cccc}
& 0 & 1 & 2 \\
\hline
4 &  3 & 1 &- \\
5 &  1& 2 &1 \\
6 &  - & 1 &- \\
7 &  - & - &- \\
8 &  1 & - &- \\
9 &  - & - &- \\
10 &- & 1 &- \\
11 & 1 & 3 &1 \\
12 & - & 1 &2 
\end{array}
$$
\end{example}

\begin{example}\label{ExampleCurveCodim2}
Let $\Omega \subseteq \mathbb{P}^3$ be a 1-dimensional complete intersection of degrees  $(4,4)$,
and $\Gamma \subseteq \Omega $ a curve with Hilbert series
$$\HS(A /I_\Gamma) = \frac{1+2t + 3t^2 +4t^3 +t^4 + t^5  - t^7 }{(1-t)^2}.$$
As before, we reduce to considering 1-dimensional algebras $R =S/I$ where  $S= \mathbb{\Bbbk}[x_1,x_2,x_3]$, $\HF(R) = (1, 3, 6, 10, 11, 12, 11, 11, 11, \ldots)$ and  $I$ contains a regular sequence of degrees $(4,4)$.
The upper bounds for the Betti table of $I_\Gamma$ obtained via the Bigatti-Hulett-Pardue Theorem and  Theorem \ref{TheoremMainInBodyPaper} are respectively

$$
\begin{array}{r|ccc}
& 0 & 1 & 2 \\
\hline
4  & 4 & 4 &1 \\
5  & 1& 2 &1 \\
6  & 3 & 5 &2 \\
7  & 1 & 2 &1 \\
8  & 1 & 2 &1 \\
9  &1 & 2 &1 \\
10  &1 & 2 &1 \\
11  & 1 & 1 &- \\
\end{array}
\hspace*{3cm}
\begin{array}{r|ccc}
& 0 & 1 & 2 \\
\hline
4 & 4 & 3 &1 \\
5 & -& - &- \\
6 & 1 & 3 &1 
\end{array}
$$
\end{example}

In the last example, we  apply Theorem \ref{TheoremMainInBodyPaper} to an ideal whose largest complete intersection 
does not satisfy our assumption on degrees.
The resulting bounds are thus between the ones of the Bigatti-Hulett-Pardue Theorem and the optimal ones predicted by Conjecture \ref{ConjectureLPP}.

\begin{example}\label{ExampleComparison3Bounds}
Let $ S= \Bbbk[x_1, x_2, x_3, x_4]$
and $I\subseteq S$ be a homogeneous ideal containing a complete intersection of degrees $(3,3,3)$ and with  
$$
\HF(S/I) = (1, 4, 10, 14, 17, 18, 17, 17, 17, \ldots).
$$
We compute the lex ideal, the $(3,3,5,\infty)$-LPP ideal, and the $(3,3,3,\infty)$-LPP ideal of $I$ in $S$ and obtain
the upper bounds for the Betti table of $I$ provided  respectively by the Bigatti-Hulett-Pardue Theorem,  Theorem \ref{TheoremMainInBodyPaper}, and Conjecture \ref{ConjectureLPP}:
  
$$
\begin{array}{r|cccc}
& 0 & 1 & 2 & 3 \\
\hline
3 &  6 & 9 &5 &1 \\
4 &  3 & 8 &7 &2 \\
5 &  5 & 12 &10 &3 \\
6 &  5 & 13 &11 &2 \\
7 &  2 & 6 &6 &2 \\
8 &  2 & 6 &6 &2 \\
9 &  2 & 5 &4 &1 \\
10 &  1 & 3 &3 &1 \\
11 &  1 & 3 &3 &1 \\
12 &  1 & 3 &3 &1 \\
13 &  1 & 3 &3 &1 \\
14 &  1 & 3 &3 &1 \\
15 &  1 & 3 &3 &1 \\
16 &  1 & 3 &3 &1 \\
17 &  1 &2 &1 &- \\
\end{array}
\hspace*{1cm}
\begin{array}{r|cccc}
& 0 & 1 & 2 & 3 \\
\hline
3 &  6 & 8 &4 &1 \\
4 &  2 & 5 &4 &1 \\
5 &  3 &6 &4 &1 \\
6 &  2 & 6 &6 &2 \\
7 &  - & 1 &1 &- \\
\end{array}
\hspace*{1cm}
\begin{array}{r|cccc}
& 0 & 1 & 2 & 3 \\
\hline
3 &  6 & 6 &4 &1 \\
4 &  - & 2 &-&- \\
5 &  - & 1 &- &- \\
6 &  1 & 3 &4 &1 \\
\end{array}
$$

\end{example}

Example \ref{ExampleComparison3Bounds} points to a  general phenomenon:

\begin{remark}
Suppose that an ideal $I\subseteq S$ contains two complete intersections $\mathfrak{f}, \mathfrak{g} $ with degree sequences $\bfd , \mathbf{e}$ such that $d_i \leq e_i$ for all $i$.
If both $\LL^{\bfd}(I)$ and $\LL^\mathbf{e}(I)$ exist, then  $\LL^{\mathbf{e}}(\LL^\mathbf{d}(I)) = \LL^\mathbf{e}(I)$,
 and by \cite[Theorem 8.1]{MM11}
we deduce that $\beta_{i,j}^S(\LL^\mathbf{d}(I)) \leq \beta_{i,j}^S(\LL^\mathbf{e}(I))$ for every $i,j$.

Furthermore,   when the field $\Bbbk$ is infinite there exists a complete intersection $\mathfrak{f} \subseteq I$ 
whose degree sequence is the smallest possible componentwise.
To see this, choose $\mathfrak{f}=(f_1, \ldots, f_c)$ with the lexicographically least possible  degree sequence,
and assume by contradiction that there exists $\mathfrak{g}=(g_1, \ldots, g_{c'})\subseteq I$ with $\deg(g_k) < \deg(f_k)$ for some $k$;
we may harmlessly assume that $c=c'=\codim (I)$.
The ideal $\frac{(g_1, \ldots, g_k) + (f_1, \ldots, f_{k-1})}{(f_1, \ldots, f_{k-1})}$ has positive codimension in the Cohen-Macaulay ring $S/(f_1, \ldots, f_{k-1})$, 
and is generated in degrees at most $ \deg(g_k)$.
By a standard prime avoidance argument, 
there exists a form $f'_k \in  I$ of degree at most $\deg(g_k)$ such that $\{f_1, \ldots, f_{k-1}, f'_k\}$ is a regular sequence;
completing this to a maximal regular sequence in $I$ we obtain a contradiction to the choice of $\mathfrak{f}$.

In conclusion, for every $I\subseteq S$ there exists a unique degree sequence $\bfd$ such that Conjecture \ref{ConjectureLPP} yields the sharpest possible bounds for the whole Betti table of $I$.
\end{remark}

\subsection*{Acknowledgments}
The first author would like to thank Bernd Ulrich for some helpful discussions on this subject.

\end{document}